\documentclass[12pt,a4paper]{amsart}
\usepackage{amsopn,amssymb,mathrsfs}

\newcommand{\cf}[1]{\ensuremath{\cofin({#1})}}

\newcommand{\closure}[1]{\ensuremath{\overline{{#1}}}}

\newcommand{\diam}[1]{\ensuremath{\sdiam{({#1})}}}

\newcommand{\httree}[1]{\ensuremath{\height({#1})}}

\newcommand{\mapping}[3]{\ensuremath{{#1}:{#2}\longrightarrow{#3}}}

\newcommand{\nat}{\mathbb{N}}

\newcommand{\rat}{\mathbb{Q}}
\newcommand{\real}{\mathbb{R}}

\newcommand{\setcomp}[2]{\ensuremath{\{{#1}\;|\;\,{#2}\}}}

\newcommand{\weakstar}{\ensuremath{w^*}}
\newcommand{\wone}{\ensuremath{\omega_1}}

\DeclareMathOperator{\card}{card} %
\DeclareMathOperator{\cofin}{cf} %
\DeclareMathOperator{\height}{ht} %
\DeclareMathOperator{\rank}{rank} %
\DeclareMathOperator{\sdiam}{diam} %

\newtheorem{thm}{Theorem}[section]
\newtheorem{cor}[thm]{Corollary}
\newtheorem{lem}[thm]{Lemma}
\newtheorem{prop}[thm]{Proposition}

\theoremstyle{definition}
\newtheorem{defn}[thm]{Definition}

\begin{document}

\title{A characterisation of compact, fragmentable linear orders}
\author{R.\ J.\ Smith}
\address{Institute of Mathematics of the AS CR,
\v{Z}itn\'{a} 25, CZ - 115 67 Praha 1, Czech Republic}
\email{smith@math.cas.cz}

\subjclass[2000]{06A05; 54D30; 46B50; 46B26}
\date{November 2008}
\keywords{Fragmentability, linear order, Radon-Nikod\'{y}m compact,
strictly convex norm}

\begin{abstract}
We give a characterisation of fragmentable, compact linearly order
spaces. In particular, we show that if $K$ is a compact,
fragmentable, linearly ordered space then $K$ is a Radon-Nikod\'{y}m
compact. In addition, we obtain some corollaries in topology and
renorming theory.
\end{abstract}

\maketitle

\section{Introduction}

Throughout this note, all topological spaces are assumed to be
Hausdorff.

\begin{defn}\label{fragrn}
Let $K$ be a compact space.
\begin{enumerate}
\item\label{frag} We say that $K$ is {\em fragmentable} if there exists a metric
$\mapping{d}{K \times K}{[0,\infty)}$ with the property that given
any non-empty set $M \subseteq K$ and $\varepsilon > 0$, there
exists an open set $U \subseteq K$ satisfying $M \cap U \neq
\varnothing$ and $d$-$\diam{M \cap U} < \varepsilon$.
\item We say that $K$ is an Radon-Nikod\'{y}m compact, or RN compact, if there
exists a metric that is lower semicontinuous on $K \times K$ and
satisfies the conditions in (\ref{frag}).
\end{enumerate}
\end{defn}

Fragmentable and RN compact spaces have been the subject of enduring
study. The paper of Namioka \cite{namioka:87} contains many of the
fundamental results on RN compact spaces. For example, Namioka
showed that the definition of RN compacta given above is equivalent
to the original definition of RN compact spaces, namely that $K$ is
RN compact if it is homeomorphic to a $\weakstar$-compact subset of
an Asplund Banach space.

The most well known unsolved problem in the theory of RN compacta is
the question of whether the continuous image of an RN compact space
is again RN compact. In \cite{aviles:07}, it is proved that if $K$
is a linearly ordered compact space and the continuous image of a RN
compact, then $K$ is RN compact. This result uses a necessary
condition for $K$ to be fragmentable. We say that a compact space
$M$ is {\em almost totally disconnected} if it is homeomorphic to
some $A \subseteq [0,1]^\Gamma$ in the pointwise topology, with the
property that if $f \in A$ then $f(\gamma) \in (0,1)$ for at most
countably many $\gamma \in \Gamma$. The class of almost totally
disconnected spaces contains all Corson compact spaces and all
totally disconnected spaces.

\begin{thm}\cite[Theorem 3]{aviles:07}\label{avilesnec}
Let $K$ be a linearly ordered fragmentable compact. Then $K$ is
almost totally disconnected.
\end{thm}

In \cite{aviles:07}, Avil\'{e}s asks whether a linearly ordered
compact space $K$ is RN compact whenever $K$ is fragmentable. In
this note, we characterise compact, fragmentable, linearly ordered
spaces.

\begin{thm}
\label{fragcharac} Let $K$ be a compact, linearly ordered space.
Then the following are equivalent.
\begin{enumerate}
\item $K$ is fragmentable;
\item there is a family $L_n$, $n \in \nat$, of compact, scattered
subsets of $K$, with union $L$, such that whenever $u,v \in K$ and
$u < v$, there exist $x,y \in L$ satisfying $u \leq x < y \leq v$.
\item $K$ is RN compact.
\end{enumerate}
\end{thm}

In doing so, we obtain Avil\'{e}s's result concerning continuous
images.

\begin{cor}\cite[Corollary 4]{aviles:07}\label{avilesrn}
Let $K$ be a compact, linearly ordered space that is also a
continuous image of a RN compact. Then $K$ is a RN compact.
\end{cor}

In fact, Corollary \ref{avilesrn} is originally stated in
\cite{aviles:07} in terms of {\em quasi-Radon-Nikod\'{y}m} compact
spaces, which we won't define here. All we need to know is that if
$K$ is a continuous image of a RN compact then it is a quasi-RN
compact, which in turn implies that $K$ is fragmentable. With this
in mind, we will see that the original statement also follows from
Theorem \ref{fragcharac}.

The proof of Theorem \ref{fragcharac}, (1) $\Rightarrow$ (2), is the
subject of Sections \ref{simplesubsets} and \ref{partitionsection}.
As a byproduct of this investigation, we obtain more results. We
denote the first uncountable cardinal by $\wone$.

\begin{prop}
\label{fragweight} Let $K$ a compact, fragmentable, linearly ordered
space, and assume that $K$ contains no order-isomorphic copy of
$\kappa$, where $\kappa$ is a regular, uncountable cardinal. Then
the topological weight of $K$ is strictly less than $\kappa$. In
particular, if $K$ contains no copy of $\wone$ then $K$ is
metrisable.
\end{prop}

The next corollary extends a theorem in \cite{talagrand:86}, which
states that $C(\wone + 1)$ admits no equivalent norm with a strictly
convex dual norm.

\begin{cor}
\label{dualr} Let $K$ be a compact, linearly ordered set, and
suppose that $C(K)$ admits an equivalent norm with strictly convex
dual norm. Then $K$ is metrisable.
\end{cor}

This leads directly to the final result.

\begin{cor}
\label{gru} Let $K$ be a compact, linearly ordered, Gruenhage space.
Then $K$ is metrisable.
\end{cor}

It is worth noting that we cannot simply demand that the union $L$
in Theorem \ref{fragcharac}, part (2), is topologically dense in
$K$. If we let $K = [0,1] \times \{0,1\}$ be the lexicographically
ordered `split interval', then $K$ is separable. However, it is well
known not to be fragmentable. We provide a proof of this at the
beginning of section \ref{simplesubsets}. Alternatively, we can use
Proposition \ref{fragweight}, since it is easy to show that $K$ is
not second countable and does not contain any uncountable well
ordered or conversely well ordered subsets. Incidentally, since $K$
is 0-dimensional, this example shows that the necessary condition of
Theorem \ref{avilesnec} is not sufficient. Moreover, it is not
possible to deduce Proposition \ref{fragweight} from Theorem
\ref{avilesnec}.

We conclude this section by proving Theorem \ref{fragcharac}, (2)
$\Rightarrow$ (3). The argument is a straightforward elaboration of
Namioka's proof that the `extended long line' is RN compact; see
\cite[Example 3.9, (b)]{namioka:87} for details. Of course, Theorem
\ref{fragcharac}, (3) $\Rightarrow$ (1), follows immediately from
Definition \ref{fragrn}. We shall use $(x,y)$ to denote both ordered
pairs and open intervals. The interpretation of the notation should
be clear from the context. We make use of the following
characterisation of RN compacta.

\begin{thm}[{\cite[Corollary 3.8]{namioka:87}}]
\label{namrn} A compact Hausdorff space $K$ is RN compact if and
only if there exists a norm bounded set $\Gamma \subseteq C(K)$ such
that
\begin{enumerate}
\item $\Gamma$ separates points of $K$, and
\item for every countable set $A \subseteq \Gamma$, $K$ is
separable, relative to the pseudo-metric $d_A$, given by
$$
d_A(x,y) \;=\; \sup\setcomp{|f(x) - f(y)|}{f \in A}
$$
for $x,y \in K$.
\end{enumerate}
\end{thm}

\begin{proof}[Proof of Theorem \ref{fragcharac}, (2) $\Rightarrow$ (3)]
Let $L_n$ be as in (2). We shall assume that $\min K$, $\max K \in
L_1$ and $L_n \subseteq L_{n+1}$ for all $n$. Let
$$
\Delta_n \;=\; \setcomp{(x,y) \in L_n^2}{x < y\mbox{ and }(x,y) \cap
L_n \mbox{ is empty}}.
$$
For each $(x,y) \in \Delta_n$, take an increasing function $f_{x,y}
\in C(K)$ such that $f(w) = 0$ for $w \leq x$ and $f(z) = n^{-1}$
for $z \geq y$. We claim that the family $f_{x,y}$, $(x,y) \in
\Delta_n$, $n \in \nat$ separates points of $K$ and, moreover, if $A
\subseteq \bigcup_{n=1}^\infty \Delta_n$ is countable, then $K$ is
$d_A$-separable, where $d_A$ is the pseudo-metric defined by
$$
d_A(u,v) \;=\; \sup \setcomp{|f_{x,y}(u) - f_{x,y}(v)|}{(x,y) \in
A}.
$$
First, let $u,v \in K$, with $u < v$. By the hypothesis and the fact
that $L_n \subseteq L_{n+1}$, there exists $n$ and $x,y \in L_n$
such that $u \leq x < y \leq v$. If we take an isolated point $w$ of
$[x,y] \cap L_n$ then at least one of the points
$$
\inf\setcomp{z \in [x,y] \cap L_n}{w < z},\quad \sup\setcomp{z \in
[x,y] \cap L_n}{z < w}
$$
is in $[x,y] \cap L_n$ and necessarily not equal to $w$. Therefore,
we can find $(x^\prime,y^\prime) \in \Delta_n$ with $x \leq x^\prime
< y^\prime \leq y$. It follows that $f_{x^\prime,y^\prime}$
separates $u$ and $v$.

Now let $A \subseteq \bigcup_{n=1}^\infty \Delta_n$ be countable. We
set $A_n = A \cap \Delta_n$ and
$$
M_n \;=\; \closure{\bigcup_{(x,y) \in A_n} \{x,y\}}.
$$
Since $M_n$ is compact, separable, scattered and linearly ordered,
it is countable. If $(x,y) \in \Delta_n$ then observe that, for
every $i \leq n$, there is a unique $(x_i,y_i) \in \Delta_i$
satisfying $x_i \leq x < y \leq y_i$. Define
$$
\Gamma_{x,y,n} \;=\; \setcomp{((q_1,r_1),\ldots,(q_n,r_n)) \in
(\rat^2)^n}{q_k < r_k \mbox{ and } \bigcap_{i=1}^n
f^{-1}_{x_i,y_i}(q_i,r_i) \neq \varnothing}
$$
and, for every $(q,r) = ((q_1,r_1),\ldots,(q_n,r_n)) \in
\Gamma_{x,y,n}$, take
$$
z_{x,y,n,q,r} \in L \cap \bigcap_{i=1}^n f^{-1}_{x_i,y_i}(q_i,r_i).
$$
We claim that the countable set
$$
D \;=\; \{\min K,\max K\} \cup \bigcup_{n=1}^\infty M_n \cup
\setcomp{z_{x,y,n,q,r}}{(x,y) \in A_n, (q,r) \in
\Gamma_{x,y,n}\mbox{ and } n \in \nat}
$$
is $d_A$-dense in $K$.

If $w \in K$ and $n \in \nat$, we find $z \in D$ such that $d_A(w,z)
< n^{-1}$. We assume $w \notin D$ and let
$$
M \;=\; \{\min K,\max K\} \cup \bigcup_{k=1}^n M_k.
$$
If $w \in (x,y)$ and $(x,y) \in A_k$ for some $k \leq n$ then let
$k$ be maximal. Otherwise, let $k = 0$. Since $M$ is closed and
$\min K,\max K \in M$, we can find $u,v \in M$, such that $(u,v)
\cap M$ is empty and $w \in (u,v)$. Assume that $k < j \leq n$ and
$(x,y) \in A_j$. We must have $f_{x,y}(u) = f_{x,y}(w) =
f_{x,y}(v)$. Indeed, if $(x,y) \in A_j$ then $x,y \in M$. Since
$(u,v) \cap M$ is empty, either $x \leq u$ and $v \leq y$, or $y
\leq u$, or $v \leq x$. However, the first possibility cannot hold
because $w \notin (x,y)$.

Now let $w \in (x,y)$, where $(x,y) \in A_k$. For $i \leq k$, take
rationals $q_i$ and $r_i$ such that $q_i < f_{x_i,y_i}(u) < r_i$ and
$r_i - q_i < n^{-1}$. Let $(x^\prime,y^\prime) \in A_i$, where $i
\leq k$. If $(x^\prime,y^\prime) \neq (x_i,y_i)$ then since
$w,z_{x,y,n,q,r} \in (x,y) \subseteq (x_i,y_i)$, we have
$$
f_{x^\prime,y^\prime}(w) \;=\; f_{x^\prime,y^\prime}(z_{x,y,n,q,r}).
$$
On the other hand, if $(x^\prime,y^\prime) = (x_i,y_i)$ then we have
ensured that
$$
|f_{x^\prime,y^\prime}(w) - f_{x^\prime,y^\prime}(z_{x,y,n,q,r})|
\;<\; n^{-1}.
$$
If $z_{x,y,n,q,r} \leq w$ then set $z = \max\{u,z_{x,y,n,q,r}\}$,
and if $w < z_{x,y,n,q,r}$ then set $z = \min\{v,z_{x,y,n,q,r}\}$.
By the construction, and the fact that the $f_{x,y}$ are increasing,
we have made sure that $|f_{x,y}(w) - f_{x,y}(z)| < n^{-1}$ whenever
$(x,y) \in \bigcup_{i=1}^n A_n$. If $(x,y) \in A_m$ and $n < m$,
then $|f_{x,y}(w) - f_{x,y}(z)| \leq m^{-1} < n^{-1}$. Therefore
$d_A(w,z) < n^{-1}$. That $K$ is RN compact now follows from Theorem
\ref{namrn}.
\end{proof}

\section{Simple subsets of trees}\label{simplesubsets}

It is a standard result in elementary analysis that every
uncountable set $H \subseteq \real$ contains an uncountable subset
$E$, with the property that each $x \in E$ is a two-sided
condensation point of $E$; that is, given $x \in E$ and $\varepsilon
> 0$, both $(x - \varepsilon,x) \cap E$ and $(x,x + \varepsilon)
\cap E$ are uncountable. As explained in \cite{aviles:07}, the
abundance of condensation points in each uncountable subset of
$\real$ can be used to show that the split interval $K = [0,1]
\times \{0,1\}$ is not fragmentable. It is worth repeating the
argument here. If $d$ is a metric on $K$, then there exists an
uncountable subset $H \subseteq [0,1]$ with the property that
$d((x,0),(x,1)) \geq n^{-1}$ for all $x \in H$. If $E \subseteq H$
is as above, then for every open subset $U \subseteq K$ such that $U
\cap (E \times \{0,1\})$ is non-empty, $d$-$\diam{U \cap (E \times
\{0,1\})} \geq n^{-1}$.

In this note, we shall consider subsets of more general linear
orders which behave similarly to uncountable subsets of $\real$ in
this way. In order to do so, we first define and investigate a
family of subsets of trees. Linear orders share a long and close
relationship with trees, and trees will feature strongly in what
follows. For convenience, we lay down some of the basic definitions.
A partially ordered set $(T,<)$ is a {\em tree} if the set of
predecessors of any given element of $T$ is well ordered. If $x,z
\in T$, we define the interval $(x,z] = \setcomp{y \in T}{x < y \leq
z}$. We introduce elements $0$ and $\infty$, not in $T$, with the
property that $0 < x < \infty$ for all $x \in T$, and define the
intervals $(0,z)$, $[x,\infty)$ and $[x,z]$ in the obvious manner.
In general, we say that $s \subseteq T$ is an {\em interval} if $y
\in s$ whenever $x \leq y \leq z$ and $x,z \in s$. We let
$\httree{x}$ be the order type of $(0,x)$ and, if $\alpha$ is an
ordinal, we let $T_\alpha$ be the level of $T$ of order $\alpha$;
that is, the set of all $x \in T$ satisfying $\httree{x} = \alpha$.
The {\em interval topology} of $T$ takes as a basis all sets of the
form $(x,z]$, where $x \in T \cup \{0\}$ and $z \in T$. This
topology is locally compact and scattered. We shall say that a
subset of $T$ is open if it is so with respect to this topology. If
$x \in T$, we let $x^+$ be the set of immediate successors of $x$.
We say that $y \in T$ is a {\em successor} if $y \in x^+$ for some
$x$ or, equivalently, $y \in T_{\xi + 1}$ for some $\xi$. In this
note, we shall only consider trees with the {\em Hausdorff}
property, i.e., if $A \subseteq T$ is non-empty and totally ordered,
then $A$ has at most one minimal upper bound. If $T$ is Hausdorff
then the interval topology on $T$ is Hausdorff in the usual,
topological sense. If $H$ is a subset of a tree $T$, the map
$\mapping{\pi}{H}{T}$ is {\em regressive} if $\pi(x) < x$ for all $x
\in H$. If $\alpha$ is a limit ordinal with cofinality $\kappa =
\cf{\alpha}$, we say that $(\alpha_\xi)_{\xi < \kappa}$ is a {\em
cofinal sequence} of ordinals if it is strictly increasing and
converges to $\alpha$.

\begin{defn}
Let $\alpha$ be a limit ordinal. We say that $H \subseteq T_\alpha$
is {\em simple} if, for some cofinal sequence $(\alpha_\xi)_{\xi <
\kappa}$, there is a injective, regressive map
$\mapping{\pi}{H}{\bigcup_{\xi < \kappa} T_{\alpha_\xi}}$.
\end{defn}

We will tie trees to compact linear orders in the next section.
Before doing so, we explore the properties of simple sets. Firstly,
it is worth noting that the choice of cofinal sequence in the
definition of simple sets is not important.

\begin{lem}
\label{cofinal} If $(\alpha_\xi)_{\xi < \kappa}$ and
$(\alpha^\prime_\xi)_{\xi < \kappa}$ are cofinal sequences, and
$\mapping{\pi}{H}{\bigcup_{\xi < \kappa} T_{\alpha_\xi}}$ is an
injective, regressive map, then there exists an injective,
regressive map $\mapping{\pi^\prime}{H}{\bigcup_{\xi < \kappa}
T_{\alpha^\prime_\xi}}$.
\end{lem}

\begin{proof}
By taking a subsequence of $(\alpha^\prime_\xi)_{\xi < \kappa}$ if
necessary, we can assume that $\alpha_\xi \leq \alpha^\prime_\xi$
for all $\xi < \kappa$. If $x \in H$ and $\pi(x) \in
T_{\alpha_\xi}$, let $\pi^\prime(x)$ be the unique element of
$[\pi(x),x] \cap T_{\alpha^\prime_\xi}$. Now suppose that
$\pi^\prime(x) = \pi^\prime(y)$. It follows that $\pi(x)$ and
$\pi(y)$ are comparable, and since $\pi^\prime(x)$ and
$\pi^\prime(y)$ share the same level, so do $\pi(x)$ and $\pi(y)$.
Therefore, $\pi(x) = \pi(y)$, giving $x = y$.
\end{proof}

The next result reveals an important permanence property of simple
sets.

\begin{prop}
\label{simplepermanent} Let $\alpha$ be a limit ordinal. Suppose
that $H \subseteq T_\alpha$, $(\alpha_\xi)_{\xi < \kappa}$ is a
cofinal sequence, and $\mapping{\pi}{H}{\bigcup_{\xi < \kappa}
T_{\alpha_\xi}}$ is a regressive mapping, with the property that
every fibre of $\pi$ is simple. Then $H$ is simple.
\end{prop}

\begin{proof}
For each $w \in \bigcup_{\xi < \kappa} T_{\alpha_\xi}$, let
$\mapping{\pi_w}{\pi^{-1}(w)}{\bigcup_{\xi < \kappa}
T_{\alpha_\xi}}$ be a regressive, injective map. By Lemma
\ref{cofinal}, we can assume that $\pi(x) \leq \pi_{\pi(x)}(x)$ for
all $x \in H$. For $\xi \leq \xi^\prime < \kappa$, define
$$
\Gamma_{\xi,\xi^\prime} \;=\; \setcomp{x \in H}{\pi(x) \in
T_{\alpha_\xi} \mbox{ and }\pi_{\pi(x)}(x) \in
T_{\alpha_{\xi^\prime}}}.
$$
The sets $\Gamma_{\xi,\xi^\prime}$, $\xi \leq \xi^\prime < \kappa$,
are pairwise disjoint, and for every $x \in H$, there exist such
ordinals with the property that $x \in \Gamma_{\xi,\xi^\prime}$.
Take $\eta \leq \eta^\prime < \kappa$. We define $\sigma(x)$ for $x
\in \Gamma_{\eta,\eta^\prime}$ in the following way. First, observe
that the set
$$
\bigcup \setcomp{(0,x] \cap [\pi_{\pi(y)}(y),y]}{y \in
\Gamma_{\xi,\xi^\prime}, \xi \leq \xi^\prime \leq \eta^\prime \mbox{
and } (\xi,\xi^\prime) \neq (\eta,\eta^\prime)}
$$
has an upper bound $w < x$. Indeed, suppose first that $\xi \leq
\xi^\prime \leq \eta^\prime$, $(\xi,\xi^\prime) \neq
(\eta,\eta^\prime)$ and $y,z \in \Gamma_{\xi,\xi^\prime}$ are such
that
$$
(0,x] \cap [\pi_{\pi(y)}(y),y] \quad\mbox{and}\quad (0,x] \cap
[\pi_{\pi(z)}(z),z]
$$
are both non-empty. Then $\pi(y),\pi(z) < x$, so they are
comparable. As they occupy the same level $T_{\alpha_\xi}$, we have
$\pi(y) = \pi(z)$. Moreover, $\pi_{\pi(y)}(y),\pi_{\pi(y)}(z) < x$
and occupy the same level $T_{\alpha_{\xi^\prime}}$, thus $y = z$,
because $\pi_{\pi(y)}$ is injective. So $(0,x] \cap
[\pi_{\pi(y)}(y),y]$ is non-empty for at most one $y \in
\Gamma_{\xi,\xi^\prime}$. Because $(\xi,\xi^\prime) \neq
(\eta,\eta^\prime)$ and $T$ is Hausdorff, this intersection has an
upper bound $w_{\xi,\xi^\prime} < x$. Since $\xi \leq \xi^\prime
\leq \eta^\prime < \kappa = \cf{\alpha}$, the required upper bound
$w$ exists. Thus we can take $\sigma(x) \in \bigcup_{\xi < \kappa}
T_{\alpha_\xi}$, satisfying $w,\pi_{\pi(x)}(x) < \sigma(x)$.

Now let $x,y \in H$. We claim that if $[\sigma(x),x] \cap
[\sigma(y),y]$ is non-empty, then $x = y$. Indeed, take such $x$ and
$y$, and let $\xi \leq \xi^\prime$, $\eta \leq \eta^\prime$ satisfy
$x \in \Gamma_{\eta,\eta^\prime}$ and $y \in
\Gamma_{\xi,\xi^\prime}$. Without loss of generality, we can assume
that $\xi^\prime \leq \eta^\prime$. If $(\xi,\xi^\prime) =
(\eta,\eta^\prime)$ then, as above, $\pi(x)$ and $\pi(y)$ are
comparable, as are $\pi_{\pi(x)}(x)$ and $\pi_{\pi(y)}(y)$. Since
both pairs occupy the same levels respectively, we get $x = y$.
Instead, if $(\xi,\xi^\prime) \neq (\eta,\eta^\prime)$ then from the
construction above, it follows that
$$
[\sigma(x),x] \cap [\sigma(y),y] \;\subseteq\; [\sigma(x),x] \cap
[\pi_{\pi(y)}(y),y]
$$
is empty. Thus we must have $(\xi,\xi^\prime) = (\eta,\eta^\prime)$
and $x = y$.
\end{proof}

The next corollary follows immediately from the proof of Proposition
\ref{simplepermanent}.

\begin{cor}
\label{pairwisedisjtintervals}Let $H \subseteq T_\alpha$ be simple.
Then for every cofinal sequence $(\alpha_\xi)_{\xi < \kappa}$, there
exists a regressive map $\mapping{\pi}{H}{\bigcup_{\xi < \kappa}
T_{\alpha_\xi}}$ with the property that $x = y$ whenever
$$
[\pi(x),x] \cap [\pi(y),y]
$$
is non-empty.
\end{cor}

Proposition \ref{simplepermanent} yields another straightforward
corollary.

\begin{cor}
\label{simpleunions} If $\alpha$ is a limit ordinal, $\kappa =
\cf{\alpha}$ and $H_\xi \subseteq T_\alpha$ is simple for all $\xi <
\kappa$, then so is the union $H = \bigcup_{\xi < \kappa} H_\xi$. In
particular, countable unions of simple subsets of $T_\alpha$ are
simple.
\end{cor}

\begin{proof}
Assume the $H_\xi$ are pairwise disjoint. Let $(\alpha_\xi)_{\xi <
\kappa}$ be a cofinal sequence and define
$\mapping{\pi}{H}{\bigcup_{\xi < \kappa} T_{\alpha_\xi}}$ by letting
$\pi(x)$ be the unique element of $(0,x] \cap T_{\alpha_\xi}$,
whenever $x \in H_\xi$. Then $\pi^{-1}(w) \subseteq H_\xi$ whenever
$w \in T_{\alpha_\xi}$.
\end{proof}

We finish this section with a final result concerning simple
subsets. We say that $W \subseteq T$ is an {\em initial part of} $T$
if $x \in W$ whenever $x \in T$, $x \leq y$ and $y \in W$.

\begin{prop}
\label{openpartition} Let $T$ be a tree and suppose that the level
$T_\alpha$ is simple for every limit ordinal $\alpha$. Then there is
a partition of $T$ consisting entirely of open intervals.
\end{prop}

\begin{proof}
The proof comes in two parts. In the first part, we prove the
following claim. Let $\mathscr{W}$ be a family of initial parts of
$T$, with union $T$, and such that, given $V,W \in \mathscr{W}$,
either $V$ is an initial part of $W$, or vice-versa. Suppose further
that each $W$ has a partition $P_W$ consisting wholly of open
intervals of $W$, with the property that if $V$ is an initial
segment of $W$ then, for every $s \in P_V$, there exists $t \in P_W$
with $s \subseteq t$. Then $T$ has a partition $P$ consisting of
open intervals only.

Define $\sim$ on $T$ by declaring that $x \sim y$ if and only if
$x,y \in s$ for some $s \in P_W$, $W \in \mathscr{W}$. This is an
equivalence relation. Symmetry is immediate. If $x \in T$ then $x
\in s$ for some $s \in P_W$ and $W \in \mathscr{W}$, so $\sim$ is
reflexive. If $x \sim y$ and $y \sim z$ then take $s \in P_V$, $t
\in P_W$ with $x,y \in s$ and $y,z \in t$. Without loss of
generality, we may assume that $V$ is an initial segment of $W$, and
so there is $u \in P_W$ with $s \subseteq u$. Since $y \in t \cap
u$, we have $x,z \in t$, so transitivity holds. If $P$ is the
corresponding partition of $T$ then it is clear that each $s \in P$
is an open interval. This completes the first part of the proof.

For the second part, for each initial segment $W_\alpha =
\bigcup_{\xi < \alpha} T_\xi$, we construct a partition
$P_{W_\alpha}$ in such a way that the resulting family satisfies the
property above. Assume $(P_{W_\xi}$, $\xi < \alpha$, have been
constructed with the property in question. If $\alpha$ is a limit
ordinal then we simply define $P_{W_\alpha}$ as in part one of the
proof. If $\alpha$ is a successor ordinal $\eta + 1$ then there are
two cases. If $\eta$ is itself a successor ordinal then all we need
to do is set
$$
P_{W_\alpha} \;=\; P_{W_\eta} \cup \setcomp{\{x\}}{x \in T_\eta}.
$$
Instead, if $\eta$ is a limit ordinal, we use Corollary
\ref{pairwisedisjtintervals} to obtain a regressive map
$\mapping{\sigma}{T_\eta}{\bigcup_{\xi < \eta} T_{\xi+1}}$ with the
property that
$$
[\sigma(x),x] \cap [\sigma(y),y] \;=\; \varnothing
$$
whenever $x \neq y$. Notice that each $\sigma(x)$ is a successor, so
$[\sigma(x),x]$ is an open interval. For each $x \in T_\eta$, let
$s_x$ be the unique element of $P_{W_\eta}$ containing $\sigma(x)$.
Now define
$$
P_{W_\alpha} \;=\; \setcomp{s_x \cup [\sigma(x),x]}{x \in T_\eta}
\cup \setcomp{s \in P_\alpha}{s \cap \bigcup_{x \in T_\eta}
[\sigma(x),x] = \varnothing}.
$$
It is straightforward to check that $P_{W_\alpha}$ has the required
property. This completes the induction and the proof.
\end{proof}

\section{Compact linear orders and partition trees}\label{partitionsection}

As mentioned Section \ref{simplesubsets}, linear orders and trees
enjoy a close relationship. We will employ the established notion of
a partition tree of a linear order. Let $K$ be a compact, linear
order. We let an interval $a \subseteq K$ be called {\em trivial} if
it contains at most one point.

\begin{defn}
\label{partitiontree} We shall say that a tree $T$ with level
$T_\alpha$ of order $\alpha$ is an {\em admissible partition tree}
of $K$ if it satisfies the following properties

\begin{enumerate}
\item every $a \in T$ is a non-trivial, closed interval of $K$;
\item $a \leq b$ if and only if $b \subseteq a$;
\item $T_0 = \{K\}$;
\item if $a \in T$ contains two elements then $a^+$ is empty;
\item\label{trivintersect} if $a \in T$ contains at least three elements then
$a$ has exactly two immediate successors $b$ and $c$, satisfying
$$
\min a \;=\; \min b \;<\; \max b \;=\; \min c \;<\; \max c \;=\;
\max a;
$$
\item\label{treelimit} if $\alpha$ is a limit ordinal then $a \in T_\alpha$ if and
only if there exist $\alpha_\xi \in T_\xi$, $\xi < \alpha$, with $a
= \bigcap_{\xi < \alpha} a_\xi$.
\end{enumerate}
\end{defn}

The next lemma states some obvious consequences of the definition
above.

\begin{lem}\label{obviousconsequences} Let $T$ be an admissible partition tree of $K$.
\begin{enumerate}
\item if $a,b \in T_\alpha$ are distinct then $a \cap b$ is trivial;
\item \label{comparable}if $a \cap b$ is non-trivial then $a,b \in T$ are comparable;
\item \label{weight} if $L = \bigcup_{a \in T} \{\min a,\max a\}$
then given $u,v \in K$, $u < v$, there exist $x,y \in M$ such that
$u \leq x < y \leq v$. In particular,
$$
\setcomp{(x,y)}{x,y \in L, x < y}
$$
is a basis for the topology of $K$;
\item \label{ordinal} if $s \subseteq T$ is totally ordered then
$\setcomp{\min a}{a \in s}$ and $\setcomp{\max a}{a \in s}$ are well
ordered and conversely well ordered respectively. Moreover, if the
cardinality $\kappa$ of $s$ is infinite, then either the cardinality
of $\setcomp{\min a}{a \in s}$, or that of $\setcomp{\max a}{a \in
s}$, is equal to $\kappa$.
\end{enumerate}
\end{lem}

\begin{proof}
We prove (1) by induction on $\alpha$. Suppose that the result holds
for all $\xi < \alpha$. First, assume $\alpha = \xi + 1$ and take
distinct $a,b \in T_\alpha$. Let $a \in a_0^+$ and $b \in b_0^+$ for
some $a_0,b_0 \in T_\xi$. If $a_0 = b_0$ then by Definition
\ref{partitiontree}, part (\ref{trivintersect}), $a \cap b$ is
trivial. Otherwise, by the inductive hypothesis, $a \cap b \subseteq
a_0 \cap b_0$ is trivial. Now assume that $\alpha$ is a limit
ordinal, with $a,b \in T_\alpha$ distinct. Take sequences
$(a_\xi),(b_\xi)_{\xi < \alpha}$ as in Definition
\ref{partitiontree}, part (\ref{treelimit}). If $a_\xi = b_\xi$ for
all $\xi$ then $a = b$, thus $a_\xi \neq b_\xi$ for some $\xi$. It
follows that $a \cap b \subseteq a_\xi \cap b_\xi$. This completes
the proof of (1).

To prove (2), take $a,b \in T$ with $a \cap b$ non-trivial. Now let
$a_0 \leq a$ and $b_0 \leq b$ with $a_0$ and $b_0$ in the same
level, and either $a_0 = a$ or $b_0 = b$. Then $a_0 \cap b_0$ is
non-trivial, so they must be equal. (2) follows.

For (3), let $u,v \in K$ with $u < v$. Let $a$ be the greatest
element of $T$ containing $u$ and $v$. The existence of $a$ follows
by compactness. If $a = \{u,v\}$ is a two-element set then we are
done. Otherwise, by Definition \ref{partitiontree}, part
(\ref{trivintersect}), $a$ has distinct immediate successors $b$ and
$c$, with $x = \max b = \min c$. The maximality of $a$ ensures that
$u < x < v$. Now repeat with $x$ and $v$.

(4) If $s \subseteq T$ is totally ordered then it is well ordered.
Put
$$
E \;=\; \setcomp{\xi}{s \cap T_\xi \neq \varnothing}
$$
and given $\xi \in E$, let $a_\xi$ be the unique element of $s \cap
T_\xi$. Define also
$$
F \;=\; \setcomp{\eta \in E}{\min a_\xi < \min a_\eta \mbox{
whenever }\xi < \eta}
$$
and
$$
G \;=\; \setcomp{\eta \in E}{\max a_\xi > \max a_\eta \mbox{
whenever }\xi < \eta}.
$$
It is clear that $\setcomp{\min a}{a \in s}$ and $F$ share the same
order type, and likewise $\setcomp{\max a}{a \in s}$ has converse
order type equal to that of $G$. Moreover, $E = F \cup G$. Indeed,
if there exist $\xi,\xi^\prime < \eta$ such that $\min a_\xi = \min
a_\eta$ and $\max a_{\xi^\prime} = \max a_\eta$ then, assuming as we
can that $\xi \geq \xi^\prime$, we have $a_\xi = a_\eta$, which is
not allowed in an admissible partition tree. The cardinality
assertion follows immediately.
\end{proof}

By recursion, it is clear that an admissible partition tree exists
for every compact linear order $K$ with at least two elements.
Moreover, for any such $T$ and any branch $s \subseteq T$, we have
that $\bigcap s$ contains at most two elements. From now on, we
shall assume that all trees $T$ are admissible partition trees of
compact linear orders. Recall the discussion of the split interval
at the beginning of Section \ref{simplesubsets}. Given a compact
linear order $K$, an admissible partition tree $T$ of $K$, a limit
ordinal $\alpha$ and a non-simple subset $H \subseteq T_\alpha$, we
show that
$$
\bigcup_{a \in H} \{\min a,\max a\}
$$
behaves similarly to $[0,1] \times \{0,1\}$. This allows us to find
a necessary condition for $K$ to be fragmentable. In the next three
lemmas, we shall assume that $\alpha$ is a limit ordinal and
$(\alpha_\xi)_{\xi < \kappa}$ is a cofinal sequence.

\begin{lem}
\label{boundedregressive} Let $H \subseteq T_\alpha$. Then there
exists a regressive map $\mapping{\pi}{H}{\bigcup_{\xi < \kappa}
T_{\alpha_\xi}}$ with the property that, for all $w \in \bigcup_{\xi
< \kappa} T_{\alpha_\xi}$, we can find $b \in H$ such that the set
$$
\setcomp{\min a}{a \in \pi^{-1}(w)}
$$
is bounded above by $\min b$.
\end{lem}

\begin{proof}
If $\setcomp{\min a}{a \in H}$ has a maximum element then there is
nothing to prove. Now suppose otherwise. Then, given $a \in H$,
there exists $b \in H$ with $\max a < \min b$. Since $a$ is a limit
element, we can take $\pi(a) < a$, $\pi(a) \in \bigcup_{\xi <
\kappa} T_{\alpha_\xi}$, such that $\max \pi(a) < \min b$. Now let
$w \in \bigcup_{\xi < \kappa} T_\xi$, with $\pi(a) = w$. By
definition, there exists $b \in H$ such that $\max \pi(a) < \min b$.
If $a^\prime \in \pi^{-1}(w)$ then
$$
\min a^\prime \;<\; \max a^\prime \;\leq\; \max \pi(a^\prime) \;=\;
\max \pi(a) \;<\; \min b.
$$
\end{proof}

\begin{lem}
\label{simpleendpoints} Let $H \subseteq T_\alpha$. Define
$$
L \;=\; \setcomp{b \in H}{\mbox{there is }x_b < \min b\mbox{ such
that }\setcomp{a \in H}{a \subseteq (x_b,\min b]}\mbox{ is simple}}.
$$
Then $L$ is simple. Similarly, if
$$
R \;=\; \setcomp{b \in H}{\mbox{there is }x_b > \max b\mbox{ such
that }\setcomp{a \in H}{a \subseteq [\max b,x_b)}\mbox{ is simple}}
$$
then $R$ is simple.
\end{lem}

\begin{proof}
Suppose $L$ is not simple. For each $b \in L$, we can find $\pi(b)
\in \bigcup_{\xi < \kappa} T_{\alpha_\xi}$ such that $\pi(b) < b$
and $x_b < \min \pi(b)$. Since $\pi$ is regressive and $L$ is not
simple, by Proposition \ref{simplepermanent}, there exists $w \in
\bigcup_{\xi < \kappa} T_{\alpha_\xi}$ such that $E = \pi^{-1}(w)$
is not simple. Observe that whenever $a,b \in E$ satisfy $\min a <
\min b$, we have
$$
x_b \;<\; \min \pi(b) \;=\; \min \pi(a) \;\leq\; \min a \;<\; \max a
\;\leq\; \min b
$$
whence $a \subseteq (x_b,\min b]$. It follows that whenever $b \in
E$, the set
$$
\setcomp{a \in E}{\min a \leq \min b}
$$
is simple. By Lemma \ref{boundedregressive}, there exists a
regressive map $\mapping{\sigma}{E}{\bigcup_{\xi < \kappa}
T_{\alpha_\xi}}$ with the property that whenever $w \in \bigcup_{\xi
< \kappa} T_{\alpha_\xi}$, the set
$$
\setcomp{\min a}{a \in \sigma^{-1}(w)}
$$
bounded above by $\min b$, for some $b \in E$. Hence
$\sigma^{-1}(w)$ is simple for all $w$. Therefore $E$ is simple by
Proposition \ref{simplepermanent}, which is a contradiction.
Consequently, $L$ is simple. It is clear that the `right hand'
version of Lemma \ref{boundedregressive} holds, thus $R$ is also
simple.
\end{proof}

\begin{lem}
\label{condensation} Suppose that $H \subseteq T_\alpha$ is
non-simple. Then there exists a subset $C \subseteq H$ with the
property that whenever $c \in C$, $x,y \in K$ and $x < \min c < \max
c < y$, the sets
$$
\setcomp{a \in C}{a \subseteq (x,\min c]} \quad\mbox{and}\quad
\setcomp{a \in C}{a \subseteq [\max c,y)}
$$
are both non-simple.
\end{lem}

\begin{proof}
If $M \subseteq T_\alpha$ is non-simple, let $C_M = M\setminus (L
\cup R)$, where $L$ and $R$ are defined as in Lemma
\ref{simpleendpoints}. We know that $C_M$ is non-simple by Lemma
\ref{simpleendpoints} and Corollary \ref{simpleunions}. Put $C =
C_H$. Let $c \in C$ and $x < \min c$. If we set
$$
M = \setcomp{a \in H}{a \subseteq (x,\min c]}
$$
then $M$ is non-simple. We can see that $C_M$, which is non-simple,
is a subset of $C \cap M = \setcomp{a \in C}{a \subseteq (x,\min
c]}$. Likewise, if $\max c < y$ then $\setcomp{a \in C}{a \subseteq
[\max c,y)}$ is non-simple.
\end{proof}

This allows us to give a necessity (and sufficient) condition for
the fragmentability of $K$ in terms of admissible partition trees.

\begin{prop}
\label{fragsimple} If $K$ is a compact, fragmentable, linearly
ordered set and $T$ is any admissible partition tree of $K$, then
there is a partition of $T$ consisting entirely of open intervals.
\end{prop}

\begin{proof}
The first thing to show is that if $K$ is fragmentable then
$T_\alpha$ is simple for every admissible partition tree $T$ of $K$
and limit ordinal $\alpha$. Assume that $H \subseteq T_\alpha$ is
non-simple and let $d$ be a metric on $K$. Set
$$
H_n \;=\; \setcomp{a \in H}{d(\min a,\max a) \geq n^{-1}}.
$$
By Corollary \ref{simpleunions}, $G = H_n$ is non-simple for some
$n$. Let
$$
E = \bigcup_{c \in C} \{\min c,\max c\}
$$
where $C = C_G\setminus\{\min K,\max K\}$ and $C_G$ is as in Lemma
\ref{condensation}. Suppose that $U \cap E$ is non-empty. If $\min c
\in U$ for some $c \in C$, then there exists $x < \min c$ with
$(x,\min c] \subseteq U$. From Lemma \ref{condensation}, we know
that $b \subseteq U$ for some $b \in C$, whence $\diam{U \cap E}
\geq n^{-1}$. If $\max c \in U$ for some $c \in C$ then we reach the
same conclusion. Therefore $d$ does not fragment $K$. Since the
metric was arbitrary, we deduce that $K$ is not fragmentable. To
finish the proof, use Proposition \ref{openpartition}.
\end{proof}

This result allows us to complete the proof of Theorem
\ref{fragcharac}.

\begin{proof}[Proof of Theorem \ref{fragcharac}, (1) $\Rightarrow$ (2)]
Let $K$ be a compact, fragmentable, linearly ordered space. We show
that if $T$ is any admissible partition tree and
$$
L \;=\; \bigcup_{a \in T} \{\min a,\max a\}
$$
is as in Lemma \ref{obviousconsequences}, part (\ref{weight}), then
$L$ is a countable union of compact, scattered subsets.

Let $T$ be an admissible partition tree. By Proposition
\ref{fragsimple}, let $P$ be a partition of $T$ consisting entirely
of open intervals. If $a \in T$ then the set
$$
\setcomp{s \in P}{(0,a] \cap s \mbox{ is non-empty}}
$$
is non-empty and finite by compactness, and the fact that $P$ is a
partition. We define $\rank a$ to be the cardinality of this set.
Define
$$
L_n \;=\; \bigcup_{\rank a \leq n} \{\min a,\max a\}.
$$
For convenience, we set $L_0 = \{\min K,\max K\}$. We prove by
induction on $n \geq 0$ that $L_n$ is closed and scattered. For each
$n$, define
$$
\Delta_n \;=\; \setcomp{(x,y) \in L_n^2}{x < y\mbox{ and }(x,y) \cap
L_n \mbox{ is empty}}
$$
as in the proof of Theorem \ref{fragcharac}, (2) $\Rightarrow$ (3).
Assuming that $L_n$ is closed and scattered, and prove that
$L_{n+1}$ shares these properties by showing that
$$
L_{n+1}\setminus L_n = \bigcup_{(x,y) \in \Delta_n} (x,y) \cap
L_{n+1}
$$
and that each such set $(x,y) \cap L_{n+1}$ is scattered and closed
in $(x,y)$. Let $w \in L_{n+1}\setminus L_n$. There exists $b \in T$
with $\rank b = n+1$, such that $w$ is an endpoint of $b$. If $b \in
s \in P$ then $(0,b]\setminus s$ is a closed, bounded interval, so
has a greatest element $a$, of rank $n$. Let $a$ have immediate
successors $c$ and $d$, with
$$
\min a \;=\; \min c \;<\; \max c \;=\; \min d \;<\; \max d \;=\;
\max a.
$$
Without loss of generality, we can assume that $c
\leq b$. Necessarily, $\rank c = n+1$. There are two cases: $\rank d
= n$ or $\rank d = n+1$. If $\rank d = n$ then let $x = \min a =
\min c$ and $y = \max c = \min d$. Since $b \subseteq c$ and $w
\notin L_n$, we have $x < w < y$. If $v \in (x,y)$ and $v$ is an
endpoint of some $e \in T$, then by Lemma \ref{obviousconsequences},
part (\ref{comparable}), $c$ and $e$ are comparable, and moreover $c
\leq e$. Since $\rank c = n+1$, we have $v \notin L_n$. This means
that $(x,y) \in \Delta_n$. Moreover, if $v \in L_{n+1}$ then since
$c \leq e$ and $\rank e = n+1 = \rank c$, we have $e \in s$.
Conversely, if $v \in (x,y)$ is an endpoint of some $e \in s$, then
$v \in L_{n+1}$. Therefore
$$
(x,y) \cap L_{n+1} \;=\; (x,y) \cap \bigcup_{e \in s} \{\min e, \max
e\}.
$$
Since $s$ is a closed interval, we see from Lemma
\ref{obviousconsequences}, part (\ref{ordinal}) and Definition
\ref{partitiontree}, part (\ref{treelimit}), that $(x,y) \cap
L_{n+1}$ is scattered and closed in $(x,y)$. If $\rank d = n+1$ then
let $x = \min a$ and $y = \max a$. We use a similar argument to show
that $(x,y) \in \Delta_n$ and $(x,y) \cap L_n$ is scattered and
closed in $(x,y)$.
\end{proof}

We finish this section with proofs of Proposition \ref{fragweight}
and Corollaries \ref{avilesrn}, \ref{dualr} and \ref{gru}.

\begin{proof}[Proof of Proposition \ref{fragweight}]
Let $T$ be an admissible partition tree of $K$. Suppose that
$T_\alpha$ is non-empty for all $\alpha < \kappa$. If $\alpha <
\kappa$ is a limit ordinal, note that, as $T_\alpha$ is simple by
Proposition \ref{fragsimple}, we have $\card{T_\alpha} \leq
\card{\bigcup_{\xi < \alpha} T_\xi}$. By a simple transfinite
induction, it follows that $\card{T_\alpha} < \kappa$ for all
$\alpha < \kappa$. Now, for every limit $\alpha < \kappa$, choose
$a_\alpha \in T_\alpha$. Since $T$ splits into a partition $P$ of
open sets, there exists $\sigma(a_\alpha) < a_\alpha$ with
$\sigma(a_\alpha),a_\alpha \in s_\alpha \in P$. We obtain a
regressive map $\mapping{\tau}{L}{\kappa}$ by setting $\tau(\alpha)
= \httree{\sigma(a_\alpha)}$, where $L$ is the set of limit ordinals
in $\kappa$. By the pressing down lemma, there is an ordinal $\xi$
and a stationary set $E \subseteq L$ such that $\tau(\alpha) = \xi$
for all $\alpha \in E$. It is well known in set theory that the
union of strictly less than $\kappa$ non-stationary subsets of
$\kappa$ is again non-stationary. Thus, as $\card{T_\xi} < \kappa$,
we conclude that there is $w \in T_\xi$ and a stationary subset $F$
of $E$ such that $\sigma(a_\alpha) = w$ for all $\alpha \in F$.
Being stationary, $F$ is unbounded in $\kappa$, and since $w \in
s_\alpha$ for all $\alpha \in F$, we obtain an interval of length
$\kappa$ in $T$. Therefore $K$ contains a copy of $\kappa$ by Lemma
\ref{obviousconsequences}, part (\ref{ordinal}). It follows that if
$K$ contains no copy of $\kappa$, every admissible partition tree
$T$ of $K$ has height strictly less than $\kappa$. This, together
with the fact that $\card{T_\alpha} < \kappa$ for all $\alpha <
\kappa$, means that $\card{T} < \kappa$, so the weight of $K$ is
strictly less than $\kappa$ by Lemma \ref{obviousconsequences}, part
(\ref{weight}).
\end{proof}

\begin{proof}[Proof of Corollary \ref{avilesrn}]
If $K$ is the continuous image of a RN compact then it is the
continuous image of a fragmentable compact and thus fragmentable by
\cite[Proposition 2.8]{ribarska:87}. Hence, by Theorem
\ref{fragcharac}, $K$ is RN compact.
\end{proof}

\begin{proof}[Proof of Corollary \ref{dualr}]
By \cite[Theorem 1.1]{ribarska:92}, $K$ is fragmentable. By
\cite{talagrand:86}, $K$ contains no copy of $\wone$. Thus $K$ is
metrisable by Proposition \ref{fragweight}.
\end{proof}

\begin{proof}[Proof of Corollary \ref{gru}]
By \cite[Theorem 7]{smith:08}, if $K$ is Gruenhage then $C(K)$
admits an equivalent norm with a strictly convex dual norm. Now
apply Corollary \ref{dualr}.
\end{proof}

\bibliographystyle{amsplain}

\end{document}